\documentclass[12pt]{amsart}

\usepackage{amssymb,amsfonts,amsthm,amsmath,amscd}
\usepackage{graphics}
\usepackage{graphicx}
\usepackage{supertabular}
\usepackage{tikz}
\usepackage{multirow}
\usepackage{enumerate}

\usetikzlibrary{calc, shapes, backgrounds }
\usepackage{pgfplots,pgfplotstable}
\pgfplotsset{width=7cm}%,compat=1.12}

\theoremstyle{plain}
\newtheorem{theorem}{Theorem}

\newtheorem{proposition}{Proposition}
\newtheorem{lemma}{Lemma}
\newtheorem{corollary}{Corollary}

\theoremstyle{definition}

\newtheorem{conjecture}{Conjecture}

\newtheorem{remark}{Remark}

\def\G{\mathcal{G}}
\def\P{\mathcal{P}}
\def\N{\mathcal{N}}

\def\A{\mathcal{A}}
\def\B{\mathcal{B}}

\begin{document}

% A short title is not required, but if needed use:
% \title[short title]{full title}
\title{Three-pile Sharing Nim and the quadratic time winning strategy}

\author{Nhan Bao Ho}
\address{Department of Mathematics, La Trobe University, Melbourne, Australia 3086}
\email{nhan.ho@latrobe.edu.au, nhanbaoho@gmail.com}

% For each additional author, add another set of
% \author, \address, and \email commands
%
%\author{}
%\address{}
%\email{}

%\date{(date1), and in revised form (date2).}

% \thanks entries are to acknowledge grants. You may combine
% all acknowledgments into one \thanks entry, or may use
% multiple \thanks entries. They generate footnotes without
% tags, so you must be explicit about which authors are
% thanking whom.
\thanks{}

\begin{abstract}
We study a variant of 3-pile {\sc Nim} in which a move consists of taking tokens from one pile and, instead of removing then,
topping up on a smaller pile provided that the destination pile does not have more tokens then the source pile after the move.
We discover a situation in which each column of two-dimensional array of Sprague-Grundy values is a palindrome.
We establish a formula for $\P$-positions by which winning moves can be computed in quadratic time.
We prove a formula for positions whose Sprague-Grundy values are $1$ and estimate the distribution of those positions whose nim-values are $g$.
We discuss the periodicity of nim-sequences that seem to be bounded.
\end{abstract}

\maketitle

%%%%%%%%%%%%%%%%%%%%%%%%%%%%%%%%%%%%%%%%%%%%%%%%%%
%%%%%%%%%%%%%%%%%%%%%%%%%%%%%%%%%%%%%%%%%%%%%%%%%%
\section{Introduction} \label{S.int}

%%%%%%%%%%%%%%%%%%%%%%%%%%%%%%%%%%%%%%%%%%%%%%%%%%
% the motivation

In a supermarket, it is the usual practice that bottles of milk with a shorter expiration date are placed in the front (the customer side) of shelves and those with a longer expiration date
are placed at the back (the store side).  One day, two staff challenge each other with the following game. There  are a different number of bottles with a short expiry date in rows on the shelf (Figure \ref{F.Milk} (a)).  The staff need to make sure that there is roughly an equal number of bottles in each row to  increase the likelihood that customers will take the older milk in the front rather than the fresher milk which will be placed at the back of the shelf.  The two staff take turns to move a bottle from a row which has a larger number of bottles and move it to a row which has a smaller number of bottles, such that the size of the destination row does not exceed that of the source row.  The staff member who makes the final move, that is, when the difference between the row with the highest number of bottles and the row with the lowest number of bottles is at most one,  wins. Figure \ref{F.Milk} (b) gives a possible ending state of the game in Figure \ref{F.Milk} (a). We will analyze this game in this paper.

\begin{figure}[ht]
\begin{center}
\begin{tikzpicture}

\draw[step=.5cm, gray,very thin] (0,0) grid (2,5);

\shade[ball color=black] (.25,.25) circle (.2cm);
\shade[ball color=black] (.25,.75) circle (.2cm);
\shade[ball color=black] (.25,1.25) circle (.2cm);

\shade[ball color=black] (.75,.25) circle (.2cm);
\shade[ball color=black] (.75,.75) circle (.2cm);

\shade[ball color=black] (1.25,.25) circle (.2cm);
\shade[ball color=black] (1.25,.75) circle (.2cm);
\shade[ball color=black] (1.25,1.25) circle (.2cm);
\shade[ball color=black] (1.25,1.75) circle (.2cm);
\shade[ball color=black] (1.25,2.25) circle (.2cm);

\shade[ball color=black] (1.75,.25) circle (.2cm);

%%%%%%%%%%%%%%%%%%%%%%%%%%%%%%%%
% after
\node at (3,3) {$\rightarrow$};

\draw[step=.5cm, gray,very thin] (4,0) grid (6,5);
\draw [gray,very thin] (4,0) -- (4,5);
\shade[ball color=black] (4.25,.25) circle (.2cm);
\shade[ball color=black] (4.25,.75) circle (.2cm);
\shade[ball color=black] (4.25,1.25) circle (.2cm);

\shade[ball color=black] (4.75,.25) circle (.2cm);
\shade[ball color=black] (4.75,.75) circle (.2cm);

\shade[ball color=black] (5.25,.25) circle (.2cm);
\shade[ball color=black] (5.25,.75) circle (.2cm);
\shade[ball color=black] (5.25,1.25) circle (.2cm);

\shade[ball color=black] (5.75,.25) circle (.2cm);
\shade[ball color=black] (5.75,.75) circle (.2cm);
\shade[ball color=black] (5.75,1.25) circle (.2cm);

\node at (1,5.5) {store side};
\node at (1,-.5) {customer side};

\node at (1,-1.5) {(a): before sorting};
\node at (5,-1.5) {(b): after sorting};

\end{tikzpicture}
\caption{A beginning and a possible ending state of the game.} \label{F.Milk}
\end{center}
\end{figure}
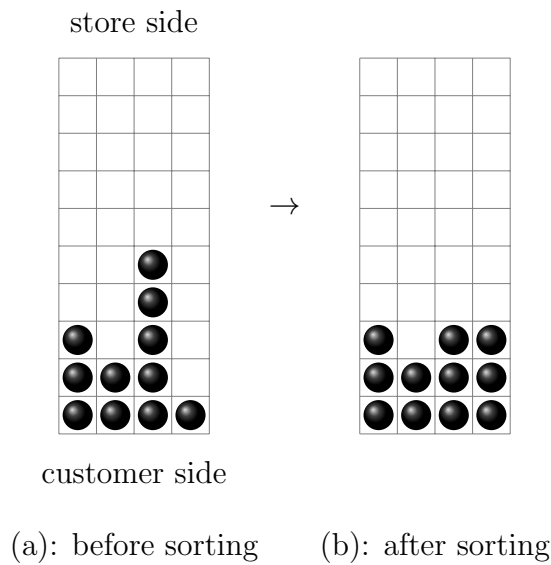

%%%%%%%%%%%%%%%%%%%%%%%%%%%%%%%%%%%%%%%%%%%%%%%%%%
% combinatorial game
A combinatorial game is played by two players in which the players move alternately, following some given rules of moves. The players are aware of the situation of the game after each move. There is no element of luck (for example, dealing cards). In the normal convention, the player who cannot find a legal move in her turn loses. A game is short if it has finitely many positions and each position can be visited no more than one time. A game is impartial if the two players have the same legal moves from every position. In this paper, games are short and impartial. The reader can find a comprehensive theory of combinatorial games, especially impartial games, in \cite{ww, Con01}.

%%%%%%%%%%%%%%%%%%%%%%%%%%%%%%%%%%%%%%%%%%%%%%%%%%
% Nim and variants
The game of {\sc Nim}, analyzed by Bouton \cite{Bouton}, provides a typical example of combinatorial games. {\sc Nim} is played with a finite row of piles of tokens. A move consists of choosing one pile and removing an arbitrary number of tokens from that pile.

Several variants of {\sc Nim} have been studied in the literature. For examples, in the game of {\sc End-Nim} \cite{Alb01}, from a row of piles of tokens, a move can remove tokens from only either of the end piles. This is an example of restriction of {\sc Nim}. In {\sc Wythoff} \cite{Wyt}, played with two piles, one can remove tokens from one pile as in {\sc Nim} or remove an equal number of tokens from both piles. {\sc Wythoff} is an example of extension of {\sc Nim}.

In this paper, we examine a 3-pile variant of {\sc Nim} in which the players move tokens from one pile to some other without reducing the total of tokens. The rule of game is as follows. There are three piles of tokens. In her turn, the player takes some tokens from one pile and adds them to one of other piles provided that after the move the receiving pile does not have more tokens than the source pile. The player who makes the last move wins. We call this game {\sc Sharing Nim}.

By definition there are three types of moves for {\sc Sharing Nim} as follows from $(a,b,c)$ with $a \leq b \leq c$:
\begin{enumerate} \itemsep0em
\item $(a,b,c) \to (a+k,b-k,c)$ with $k \leq (b-a)/2$; or
\item $(a,b,c) \to (a+k,b,c-k)$ with $k \leq (c-a)/2$; or
\item $(a,b,c) \to (a,b+k,c-k)$ with $k \leq (c-b)/2$.
\end{enumerate}

%%%%%%%%%%%%%%%%%%%%%%%%%%%%%%%%%%%%%%%%%%%%%%%%%%
% N- and P-position
A position is called an {\it $\N$-position} if the player about to move next has a plan of moves to win.
Otherwise, we have a {\it $\P$-position}.
%The understanding of the set of $\P$-positions of a game is almost equivalent to the understanding the strategy to play the game.
We also use $\P$ as the set of $\P$-positions.
A position without available move is a terminal position.

%%%%%%%%%%%%%%%%%%%%%%%%%%%%%%%%%%%%%%%%%%%%%%%%%%
% \G

%A position $q$ is said to be a follower of a position $p$ if there exists a move from $p$ to $q$.
For a subset $S$ of nonnegative integers, $\operatorname{mex}(S)$ is the least nonnegative integer not in $S$.
By this, $\operatorname{mex}(\emptyset) = 0$ (corresponding to terminal position).
The {\it Sprague-Grundy value} $\G(p)$, also known as {\it nim-value}, of a position $p$ is defined recursively as follows:
$\G(p) = \operatorname{mex}\{\G(q) \mid  \text{ there exists a move from $p$ to $q$}\}$. We call $k$-position a position whose nim-value is $k$. One can see that a position is a $\P$-position if and only if it is a $0$-position.

%%%%%%%%%%%%%%%%%%%%%%%%%%%%%%%%%%%%%%%%%%%%%%%%%%
% sum
Nim-values are essential for the understanding of those games that can be decomposed into separate components such that each of which is also a game. We call such cases sums of games. In other words, a {\it sum} $G+H$ of two given games $G$ and $H$ is a game in which a move can be made in either $G$ or $H$ following exactly the rules for chosen game. The sum game ends if there is no move available from both $G$ and $H$. It was showed independently by Sprague \cite{Spr36, Spr37} and Grundy \cite{Gru39} that $\G(G+H) = \G(G) \oplus \G(H)$ in which $\oplus$, call {\it nim-addition}, is the addition in binary number system without carrying, meaning $1 \oplus 1 = 0$. In particular, $G+H$ is a $\P$-positions if and only if $\G(H) = \G(G)$.

For example, the game of {\sc Nim} with multiple piles is the sum of single piles and the nim-value of position consisting of $n$ piles $a_1, a_2, \ldots, a_n$ is $\G(a_1) \oplus \G(a_2) \oplus \cdots \oplus \G(a_n)$.

%%%%%%%%%%%%%%%%%%%%%%%%%%%%%%%%%%%%%%%%%%%%%%%%%%
% simplification
We start with a remark on the simplification of positions of 3-pile {\sc Sharing Nim}. By the definition of the game, we have
\begin{lemma} \label{Lem.sim}
For positive integers $a, b$, and $c$ with $a \leq b \leq c$, $\G(a,b,c) = \G(0,b-a,c-a)$. We also write $(a,b,c) \equiv (0,b-a,c-a)$.
\end{lemma}
Thus, it is enough to study positions $(0,a,b)$. We also assume that $a \leq b$ in that notation.
\begin{remark}
By Lemma \ref{Lem.sim}, we will use $(0,a,b)$ to denote the class of all positions $(c, a+c, b+c)$.
\end{remark}

%%%%%%%%%%%%%%%%%%%%%%%%%%%%%%%%%%%%%%%%%%%%%%%%%%
% table of values

\begin{table}[ht]
\begin{center}
\begin{tabular}{c|cccccccccccccccccccc}
16     &  &  &  &  &  &  &  &  &  &  &  &  &  &  &  &   &0 \\
15     &  &  &  &  &  &  &  &  &  &  &  &  &  &  &  &0  &11 \\
14     &  &  &  &  &  &  &  &  &  &  &  &  &  &  &1 &10 &10 \\
13     &  &  &  &  &  &  &  &  &  &  &  &  &  &0 &9 &10 &11 \\
12     &  &  &  &  &  &  &  &  &  &  &  &  &0 &2 &9 &8  &9 \\
11     &  &  &  &  &  &  &  &  &  &  &  &0 &7 &1 &7 &5  &11  \\
10     &  &  &  &  &  &  &  &  &  &  &1 &7 &8 &8 &4 &5  &11  \\
9     &  &  &  &  &  &  &  &  &  &0 &6 &2 &8 &9 &9 &5  &6  \\
8      &  &  &  &  &  &  &  &  &3 &6 &6 &7 &6 &9 &7 &10 &7  \\
7      &  &  &  &  &  &  &  &0 &5 &1 &3 &7 &8 &4 &10&10 &6  \\
6      &  &  &  &  &  &  &1 &4 &5 &4 &4 &5 &5 &4 &7 &5  &11  \\
5      &  &  &  &  &  &0 &4 &2 &5 &6 &6 &5 &8 &9 &9 &5  &11  \\
4      &  &  &  &  &0 &3 &2 &3 &3 &6 &4 &7 &6 &9 &4 &5  &9 \\
3      &  &  &  &0 &3 &1 &4 &3 &5 &4 &3 &7 &8 &8 &7 &8  &11 \\
2      &  &  &1 &2 &2 &1 &2 &2 &5 &1 &6 &2 &8 &1 &9 &10 &10 \\
1      &  &0 &1 &2 &3 &3 &4 &4 &5 &6 &6 &7 &7 &2 &9 &10 &11 \\
0      &0 &0 &1 &0 &0 &0 &1 &0 &3 &0 &1 &0 &0 &0 &1 &0  &0 \\
\hline
a/b    &0 &1  &2 &3  &4  &5  &6  &7  &8 &9  &10  &11 &12 &13 &14 &15  &16
 \end{tabular}
\caption{nim-values $\G(0,a,b)$ for $0 \leq a \leq  b \leq 16$}\label{T1}
\end{center}
\end{table}

%%%%%%%%%%%%%%%%%%%%%%%%%%%%%%%%%%%%%%%%%%%%%%%%%%
% Observation on Table 1
We give nim-values $\G(0,a,b)$ for $0 \leq a \leq  b \leq 16$ in Table \ref{T1} in which the entry $(a,b)$ is the nim-value $\G(0,a,b)$.
The top diagonal that contains all $(a,a)$-entries is called main diagonal.
One can see in Table \ref{T1} that each column is equal to its reverse.
For example, the $5^{th}$ column in Table \ref{T1} is $(0, 3, 1, 1, 3, 0)$.
%, illustrating the property (1). We illustrate property (2) by the {\bf bold row} and {\bf diagonal} both starting from the entry $(2,4)$.

%%%%%%%%%%%%%%%%%%%%%%%%%%%%%%%%%%%%%%%%%%%%%%%%%%
% outline of paper
We will prove this phenomenon in the next section.
We then characterize $\P$-positions and discuss related properties in Section \ref{S.P}.
It will be shown that the winning strategy can be computed in quadratic time. The strategy can also be taught to first-year students in secondary school, requiring only background on arithmetics.
Section \ref{S.P.per} discusses the periodicity of $\P$-positions and $\N$-position. Especially, we raise evidence supporting the conjecture that rows and diagonals in Table \ref{T1} seem to be bounded but not ultimately periodic.
In Section \ref{S.v1}, we give a formula for 1-positions.
We discuss the distribution of nim-values in two dimensional array in Section \ref{S.Dis}.
%In Section \ref{Prog}, we discuss the difficulty of  programming data when we expand the game to more than 4 piles.

%We discuss some further directions of in the final section.
%In Section \ref{S.Per}, we prove that each row in Table \ref{T1} is ultimately periodic before proving the periodicity of sequence of $g$-positions for a given $g$.
%Section \ref{S.Mis} analysis the mis\`{e}re version of this game.

\section{Symmetry and palindromes} \label{S.palindrome}
%\section{Finite nim-sequences forming palindromes} \label{S.palindrome}

We explain the two properties of Table \ref{T1} mentioned in Section \ref{S.int}. A sequence is a {\it palindrome} if it reads the same backward or forward. For example, $(1,2,3,2,1)$ is a palindrome. We show that for $a \leq b/2$, the sequence $(\G(0,a,b), \G(0,a+1,b), \ldots, \G(0,b-a,b) )$ is a palindrome. This explains the phenomenon of columns we have mentioned on Table \ref{T1} in Section \ref{S.int}.

\bigskip

%We first prove a symmetry property of nim-values.

%%%%%%%%%%%%%%%%%%%%%%%%%%%%%%%%%%%%%%%%%%%%%%%%%%
% SYMMETRY
\begin{theorem} \label{thm.ab}
For all $a, b$ such that $0 \leq a \leq b$, $\G(0,a,b) = \G(0,b-a,b)$.
\end{theorem}

\begin{proof}
We prove the theorem by induction on $b$. It can be checked that the theorem holds for $b = 1$ and $b = 2$.
Assume that the theorem holds for $b \leq n$ for some $n \geq 2$. We show that the theorem holds for $b = n+1$.
We prove this claim by induction on $a$.

We first prove the case $a = 1$. We show that
\[\G(0,0,n+1) = \G(0,n+1,n+1).\]
A move from either of these two positions is to take $k$ tokens from a pile of size $n+1$ and add to a pile of size 0 provided that $k \leq n+1-k$.
It remains to show that
\begin{align} \label{pro.eq.1}
\G(0,k,n+1-k) = \G(k,n+1-k,n+1).
\end{align}
By Lemma \ref{Lem.sim}, the RHS is $\G(0,n+1-2k,n+1-k)$. Since $n+1-k \leq n$, by the inductive hypothesis on $b$, (\ref{pro.eq.1}) holds.

Assume that for $b = n+1$ the theorem holds for all $a \leq m$ for some $m$ such that $0 \leq m < b$. We show that the theorem holds for $a = m+1$ or equivalently \begin{align} \label{pro.eq.2}
\G(0,m+1,n+1) = \G(0,n-m,n+1).
\end{align}
%%%%%%%%%%%%%%%%%%%%%%%%%%%%%%%%%%%%%%%%%%%%%%%%%%%%%%%%%%%%%%
%GEOMETRIC ASPECT
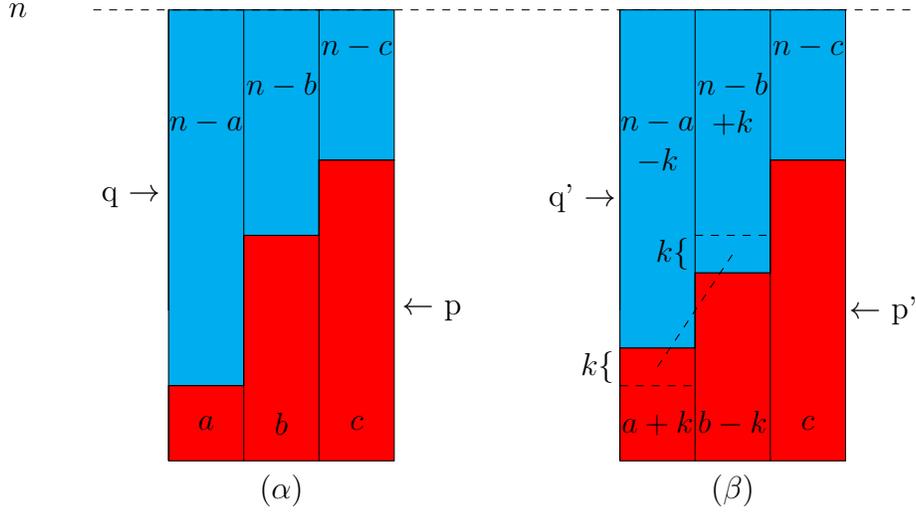
\begin{figure}[ht]
\centering
\begin{tikzpicture}

\draw [black, fill=cyan] (1,1) -- (2,1) -- (2,3) -- (3,3) -- (3,4) -- (4,4) -- (4,6) -- (1,6) -- (1,2);
\draw [black, fill=red] (1,0) -- (1,1) -- (2,1) -- (2,3) -- (3,3) -- (3,4) -- (4,4) -- (4,0) -- (1,0);

\draw [dashed] (0,6) -- (11,6);
\node at (-1,6) {$n$};

%%%
% LEFT: BEFORE MOVE
\draw (1,0) -- (1,6);
\draw (2,0) -- (2,6);
\draw (3,0) -- (3,6);

\node at (0.5,3.5) {q $\rightarrow$};
\node at (4.5,2) {$\leftarrow$ p};

\node at (1.5,.5) {$a$};
\node at (2.5,.5) {$b$};
\node at (3.5,.5) {$c$};

\node at (1.5,4.5) {$n-a$};
\node at (2.5,5) {$n-b$};
\node at (3.5,5.5) {$n-c$};

\node at (2.5,-.4) {$(\alpha)$};

%%%%%
% RIGHT: AFTER MOVE
\draw [black, fill=cyan] (7,1.5) -- (8,1.5) -- (8,2.5) -- (9,2.5) -- (9,4) -- (10,4) -- (10,6) -- (7,6) -- (7,2);
\draw [black, fill=red] (7,0) -- (7,1.5) -- (8,1.5) -- (8,2.5) -- (9,2.5) -- (9,4) -- (10,4) -- (10,0) -- (7,0);

\draw (7,0) -- (7,6);
\draw (8,0) -- (8,6);
\draw (9,0) -- (9,6);

\node at (6.5,3.5) {q' $\rightarrow$};
\node at (10.5,2) {$\leftarrow$ p'};

\node at (7.5,.5) {$a+k$};
\node at (8.5,.5) {$b-k$};
\node at (9.5,.5) {$c$};

\node at (7.5,4.5) {$n-a$}; \node at (7.5,4) {$-k$};
\node at (8.5,5) {$n-b$}; \node at (8.5,4.5) {$+k$};
\node at (9.5,5.5) {$n-c$};

\node at (6.7,1.25) {$k \{$};
\node at (7.7,2.75) {$k \{$};

\draw [dashed] (7,1) -- (8,1);
\draw [dashed] (8,3) -- (9,3);

\draw [dashed] (7.5,1.25) -- (8.5,2.75);

\node at (8.5,-.4) {$(\beta)$};

\end{tikzpicture}
\caption{$(\alpha)$: A position $p$ (red) and its complement position $q$ (blue) with $\G(p) = \G(q)$. $(\beta)$: Two moves $p \to p'$ and $q \to q'$ are symmetric. The symmetry still holds when $n$ is increased.} \label{F.com}
\end{figure}
%%%%%%%%%%%%%%%%%%%%%%%%%%%%%%%%%%%%%%%%%%%%%%%%%%%%%
Let $p = (0,m+1,n+1)$ and $q = (0,n-m,n+1)$. It suffices to prove the two following facts:
\begin{enumerate} \itemsep0em
\item [\rm{(i)}]  if there exists a move from $p$ to $p'$ then there exists a move from $q$ to $q'$ such that $\G(q') = \G(p')$;
\item [\rm{(ii)}] if there exists a move from $q$ to $q'$ then there exists a move from $p$ to $p'$ such that $\G(p') = \G(q')$.
\end{enumerate}

From some position $(0,x,y)$, there may be three types of moves, to $(k,x-k,y)$ or $(0,x+k,y-k)$ or $(k,x,y-k)$.
We will prove only one type of moves in \rm{(i)}.
The other types of moves of \rm{(i)} and \rm{(ii)} can be treated essentially the same and we leave to the reader. Using Figure \ref{F.com}, it is easy to find the corresponding moves  to $p', q'$ required in \rm{(i)} and \rm{(ii)} above. We will mention this as bellow.

Consider the move $p \to (k,m+1-k,n+1) = (0,m+1-2k,n+1-k) = p'$ where $k \leq (m+1)/2$. In Figure \ref{F.com}, this option moves $k$ tokens from the red middle pile to the left one.
Consider the move from $q$ to $q' = (0,n+1-m+k, n+1-k)$. In Figure \ref{F.com}, this option moves $k$ tokens from the blue left pile to the middle one.
Since $n+1-k \leq n$, by the inductive hypothesis on $b$, $\G(q') = \G(p')$ and so (\ref{pro.eq.2}) holds.

By the principle of mathematical induction, the theorem holds for $b = n+1$ and $a \leq b$.
Again, by the principle of mathematical induction, the theorem holds for all $b$ and $a \leq b$.
\end{proof}

Theorem \ref{thm.ab} can be explained as follows. Given position $p = (a,b,c)$ and $n \geq c$. Consider the position $q = (n-a,n-b,n-c)$. By Lemma \ref{Lem.sim} and Theorem \ref{thm.ab}, $\G(p) = \G(q)$. Note that the position $q$ obtained by adding the amount of tokens to each pile from $p$ until reaching $n$. By Lemma \ref{Lem.sim}, it is sufficient to consider the case $a = 0$ and $n = c$. We say that $q$ is the complement of the position $p$ (Figure \ref{F.com} ($\alpha$)).

%%%%%%%%%%%%%%%%%%%%%%%%%%%%%%%%%%%%%%%%%%%%%%%%%%%%
%base
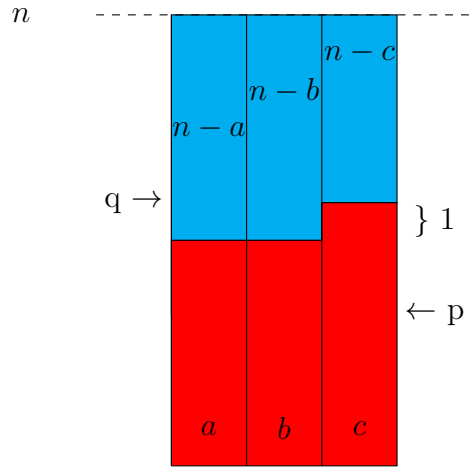
\begin{figure}[ht]
\centering
\begin{tikzpicture}

\draw [black, fill=cyan] (1,3) -- (2,3) -- (3,3) -- (3,3.5) -- (4,3.5) -- (4,6) -- (1,6) -- (1,2);
\draw [black, fill=red] (1,0) -- (1,3) -- (2,3) -- (3,3) -- (3,3.5) -- (4,3.5) -- (4,0) -- (1,0);

\draw (1,0) -- (1,6);
\draw (2,0) -- (2,6);
\draw (3,0) -- (3,6);

\node at (4.5,3.25) {$\}$ 1};

\node at (0.5,3.5) {q $\rightarrow$};
\node at (4.5,2) {$\leftarrow$ p};

\node at (1.5,.5) {$a$};
\node at (2.5,.5) {$b$};
\node at (3.5,.5) {$c$};

\node at (1.5,4.5) {$n-a$};
\node at (2.5,5) {$n-b$};
\node at (3.5,5.5) {$n-c$};

\draw [dashed] (0,6) -- (5,6);
\node at (-1,6) {$n$};

\end{tikzpicture}
\caption{A possible base case.} \label{F.com.base}
\end{figure}

\begin{remark}
One can also use Figure \ref{F.com} to prove Theorem \ref{thm.ab}. In order to prove this, we prove the formula $\G(a,b,c) = \G(n-a,n-b,n-c)$ with $n$ being fixed. The proof can be done by induction on $c-a$. The base case is when $c-a \leq 1$ (Figure \ref{F.com.base}).
\end{remark}

%%%%%%%%%%%%%%%%%%%%%%%%%%%%%%%%%%%%%%%%%%%%%%%%%%%%%%%%%%%%%%
%%%%%%%%%%%%%%%%%%%%%%%%%%%%%%%%%%%%%%%%%%%%%%%%%%%%%%%%%%%%%%

\bigskip

%%%%%%%%%%%%%%%%%%%%%%%%%%%%%%%%%%%%%%%%%%%%%%%%%%
% ISOMORPHISM
Two games $G$ and $H$ are said to be isomorphic if there is one-to-one map $f$ between position sets of $G$ and $H$ such that $x \to x'$ is a move in $G$ if and only if there exist positions $y, y'$ in $H$ such that $y \to y'$ is a move and $f(x) = y$, $f(x') = y'$. The following  follows straightforwardly from the proof of Theorem \ref{thm.ab}.

\begin{corollary} \label{C.iso}
The two game positions $(0,a,b)$ and $(0,b-a,b)$ are isomorphic.
\end{corollary}

%%%%%%%%%%%%%%%%%%%%%%%%%%%%%%%%%%%%%%%%%%%%%%%%%%
\begin{remark}
By Theorem \ref{thm.ab}, it suffices to analyze positions $(0,a,b)$ with $a \leq b/2$.
\end{remark}

%%%%%%%%%%%%%%%%%%%%%%%%%%%%%%%%%%%%%%%%%%%%%%%%%%
% PALINDROME

The following corollary means each finite column displayed in Table \ref{T1} is a palindrome.

\begin{corollary}
For $a \leq b/2$, the sequence
$$(\G(0,a,b), \G(0,a+1,b), \ldots, \G(0,b-a,b) )$$
is a palindrome.
\end{corollary}

%%%%%%%%%%%%%%%%%%%%%%%%%%%%%%%%%%%%%%%%%%%%%%%%%%
% Diagonal v.s row
It also follows from Theorem \ref{thm.ab} that in Table \ref{T1}, the row of nim-values on the right of the entry $(a,2a)$ and the diagonal (parallel to the main diagonal) of nim-values starting from this entry are identical.

\begin{corollary}
The two sequences $(\G(0,a,2a+i))_{i \geq 0}$ and $(\G(0,a+i,2a+i))_{i \geq 0}$ are identical.
\end{corollary}

%%%%%%%%%%%%%%%%%%%%%%%%%%%%%%%%%%%%%%%%%%%%%%%%%%
%%%%%%%%%%%%%%%%%%%%%%%%%%%%%%%%%%%%%%%%%%%%%%%%%%
% P-POSITIONS
\section{Characteristics of $\P$-positions} \label{S.P}
We formulate $\P$-positions and describe related properties. By this formula, we can decide in quadratic time whether or not a given position is a $\P$-position. Moreover, from an $\N$-position, one can find a winning move (moving to some $\P$-position) in quadratic time.

\begin{theorem} \label{P}
The $\P$-positions form the set
\[ \{(0,0,0), (0,0,2^{2k}(2l+1)), (0,2^{2k}(2l+1),2^{2k}(2l+1)) \mid  k \geq 0, l\geq 0\}. \]
\end{theorem}
\begin{proof}
Set
\[\A =  \{(0,0,0), (0,0,2^{2k}(2l+1)), (0,2^{2k}(2l+1),2^{2k}(2l+1)) \mid  k \geq 0, l\geq 0\}.\]
We need to prove two facts:
\begin{enumerate} \itemsep0em
\item there is no move between any two distinct positions in $\A$, and
\item from any nonterminal position not in $\A$, there exists a move that terminates in $\A$.
\end{enumerate}
For (1), first note that there is no move from either $(0,0,2^{2k}(2l+1))$ or $(0,2^{2k}(2l+1),2^{2k}(2l+1))$ to $(0,0,0)$.
There is no move from a position of the form $(0,0,2^{2k}(2l+1)$ to some other position of the same form since such a move increases one pile of size 1 and keeps the other unchanged.
There is also no move from a position of the form $(0,2^{2k}(2l+1),2^{2k}(2l+1))$ to some other position of the same form since such a move reduces one of the two biggest piles. It remains to show that
\begin{enumerate} \itemsep0em
\item [\rm{(i)}] there is no move $(0,0,2^{2k}(2l+1)) \to (0,2^{2k'}(2l'+1),2^{2k'}(2l'+1))$, and
\item [\rm{(ii)}] there is no move $(0,2^{2k}(2l+1),2^{2k}(2l+1)) \to (0,0,2^{2k'}(2l'+1))$.
\end{enumerate}
For \rm{(i)}, the only possible move that equalizes the two biggest piles is $(0,0,2^{2k}(2l+1)) \to (0,2^{2k-1}(2l+1),2^{2k-1}(2l+1))$. For \rm{(ii)}, the only move that equalizes the smallest two piles is $(0,2^{2k}(2l+1),2^{2k}(2l+1)) \to (2^{2k-1}(2l+1),2^{2k-1}(2l+1),2^{2k}(2l+1)) \equiv (0,0,2^{2k-1}(2l+1))$. Since the power of $2$ in resulted positions are $2k-1$, \rm{(i)} and \rm{(ii)} hold.

We now prove (2). Let $p = (0,a,b)$ be a position not in $\A$.
By Theorem \ref{thm.ab}, it suffices to choose $a, b$ such that $a \leq b/2$.
Set $b = a+y$. Then $p = (0,a,a+y)$ with $a \leq y$.

Consider $y-a$. There exist $k \geq 0$ and $l \geq 0$ such that either $y-a = 2^{2k}(2l+1)$ or $y-a = 2^{2k+1}(2l+1)$. For the former case, $p$ can be moved to
\begin{align*}
q &= (a,a,a+y-a) = (0,0,y-a)  \\
  &= (0,0,2^{2k}(2l+1)) \in \A.
\end{align*}
For the latter case, $p$ can be moved to
\begin{align*}
q' &= \big(a+2^{2k}(2l+1),a,a+y-(a+2^{2k}(2l+1))\big) \\
   &= (0,2^{2k}(2l+1),2^{2k}(2l+1)) \in \A.
\end{align*}
\end{proof}

Note that the two forms of $\P$-positions in Theorem \ref{P} are symmetric by Theorem \ref{thm.ab}.

\begin{remark} \label{R.P-linear}
To decide whether or not a given position is a $\P$-position, it is sufficient to test if there are two equal piles and the difference between this size and that of the remaining pile is equal to some $2^{2k}(2l+1)$. The first task involves three subtractions, requiring linear time. The second task is to shift all the bits two squares to the right repeatedly (dividing by $4$ repeatedly) before checking if the last digit is 1. Shifting rightwards two bits requires $4L$ times, with $L$ being the length ò the input. When the input contains only 1's, we need to repeat shifting $L/2$ times. Therefore, the second task requires quadratic time.
\qed
\end{remark}

\bigskip

%%%%%%%%%%%%%%%%%%%%%%%%%%%%%%%%%%%%%%%%%%%%%%%%%%
% MONKEY HAS 1/3 CHANCE TO WIN

We have the following optimization property.

\begin{remark} \label{R.opt-move}
It follows from Theorem \ref{P} and Lemma \ref{Lem.sim} that from an $\N$-position $(a,b,c)$, exactly one of the following three moves is the winning move:
\begin{enumerate} \itemsep0em
\item equalizing the two piles $a$ and $b$: moving $(b-a)/2$ tokens from pile $b$ to pile $a$, or
\item equalizing the two piles $b$ and $c$, or
\item equalizing the two piles $a$ and $c$.
\end{enumerate}
Especially, if there are two equal piles, the winning move, if it exists, can be made simply without knowledge of $\P$-positions shown in Theorem \ref{P}: equalizing one of these two piles with the third pile (the pile having different size).
\qed
\end{remark}

Thus, the winning strategy can be taught to first year students at high school with basic background on arithmetics: starting with each option in Remark \ref{R.opt-move} before dividing the difference in the size repeatedly by 4 until reaching either an odd number or an even number not divisible by 4. If the former case holds, that move is chosen. Otherwise, take another option in Remark \ref{R.opt-move}.

\begin{remark} \label{R.N-linear}
Since each option in Remark \ref{R.opt-move} can be done in linear time, by Remark \ref{R.P-linear}, from an $\N$-position, a move to some $\P$-position can be computed in quadratic time.
\qed
\end{remark}

%%%%%%%%%%%%%%%%%%%%%%%%%%%%%%%%%%%%%%%%%%%%%%%%%%
% PERMUTATION OF P
We next answer the following counting problem.
Suppose two players want to play this game with some $n$ tokens.
They need to decide the sizes of three piles.
Assume one player volunteers to do this job with the intention of letting the other move first. How many different ways that she can arrange the sizes of tokens to win. %We answer this counting problem.

\begin{proposition} \label{P.counting}
Let $P(n)$ the number of $\P$-positions involving exactly $n$ tokens. Then $P(n) = \lfloor n/3 \rfloor$ in which $\lfloor.\rfloor$ denotes the integer part.
\end{proposition}
\begin{proof}
A $\P$-position involving $n$ tokens must be $(a,a,a)$ or $(a,a,a+b)$ or $(a,a+b,a+b)$ for some $a \leq \lfloor n/3 \rfloor$ and $b = 2^{2k}(2l+1)$. Moreover, if $(a,a,a+b)$ is a $\P$-position then $(a,a+b/2,a+b/2)$ is not a $\P$-position by Theorem \ref{P}. Similarly, if $(a,a+b,a+b)$ is a $\P$-position then $(a,a,a+2b)$ is not a $\P$-position. Thus, for each $a$, there is at most one way to arrange the remaining $n-3a$ to establish a $\P$-position. Since $a \leq \lfloor n/3 \rfloor$, there are at most $\lfloor n/3 \rfloor$ $\P$-positions involving exactly $n$ tokens.

For each $a$ such that $1 \leq a \leq \lfloor n/3 \rfloor$, consider $n-3a$. Note that either $n-3a = 2^{2k}(2l+1)$ or  $n-3a = 2^{2k+1}(2l+1)$ for some $k, l \geq 0$. In the former case, we have $\P$-position $(a,a,a+2^{2k}(2l+1))$. In the latter case, we have $\P$-position $(a,a+2^{2k}(2l+1),a+2^{2k}(2l+1))$. Therefore, there are exactly $\lfloor n/3 \rfloor$ $\P$-positions.
\end{proof}

%%%%%%%%%%%%%%%%%%%%%%%%%%%%%%%%%%%%%%%%%%%%%%%%%%
%%%%%%%%%%%%%%%%%%%%%%%%%%%%%%%%%%%%%%%%%%%%%%%%%%
% PERIODICITY OF P

\section{How is not an apparently bounded nim-sequence periodic?} \label{S.P.per}

We conjecture that nim-values in each row in Table \ref{T1} are bounded.
We then discuss the periodicity of $\P$-positions and nim-values on rows of Table \ref{T1}.

%Periodicity of nim-sequences are essential to playing games since it provides a polynomial strategy.
Consider nim-values on Table \ref{T1} and suppose we can extend the table large enough. One may have following questions.
\begin{enumerate} \itemsep0em
\item Are the nim-values on each row bounded? If so, is each row ultimately periodic?
\item If the nim-values on each row are not bounded, is each row ultimately arithmetic periodic?
\end{enumerate}
Recall that a sequence $(s_i)$ is said to be ultimately periodic (resp.~arithmetic periodic) if there exist $n_0$ and $p > 0$ (resp.~and $s > 0$) such that for all $n \geq n_0$, $s_{i+p} = s_i$ (resp.~$s_{i+p} = s_i + s$). An arithmetic periodic sequence is said to be additive periodic if $s = p$.

By Theorem \ref{P}, the bottom row of Table \ref{T1}, with $a = 1$, is not ultimately arithmetic periodic and, moreover, has entries $0$ when $b$ is even. Thus, question (2) is not relevant for the bottom row and also may not relevant for other rows. Moreover, our computation leads us to the following conjecture.

\begin{conjecture} \label{C.bound}
Given $a \geq 1$, the nim-sequence $(\G(0,a,n))_{n \geq a}$ is bounded.
\end{conjecture}

For examples, the first $490$ nim-values of the sequence $(\G(0,0,n)_{n \geq 1})$ are bounded by 12 (Table \ref{T.P}, read from left to right); the first $500$ nim-values of the sequence $(\G(0,2,n)_{n \geq 3})$ are bounded by 41; the first $800$ nim-values of the sequence $(\G(0,6,n)_{n \geq 3})$ are bounded by 72. We include data for the second and the third examples in Appendix \ref{App} (Figures \ref{row3} and \ref{row7}).

If Conjecture \ref{C.bound} holds, the game displays a surprising pattern of nim-values. This property suggests that its Sprague-Grundy function would involve modulus division.

Note that multiple-pile {\sc Nim} and several of its variants have property that if the size of one pile increase, the nim-values ultimately increase. For example, on the two-dimension array of nim-values of {\sc Wythoff} whose $(i,j)$ entry is nim-value $\G(i,j)$ of {\sc Wythoff} game, every row is ultimately additive periodic \cite{Lan04} and every diagonal parallel to the main diagonal contains every nonnegative integer \cite{BF90}. Thus, rows and diagonals of nim-values for {\sc Wythoff} are not bounded.

We now move our attention to Question (1): is each row in Table \ref{T1} ultimately periodic? We start with the bottom row, studying the periodicity of the $\P$-positions. We list the first 490 values of the bottom row of Table \ref{T1} in Table \ref{T.P}.

\begin{table}[ht]
\begin{align*}
&0, 0, 1, 0, 0, 0, 1, 0, 3, 0, 1, 0, 0, 0, 1, 0, 0, 0, 1, 0, 0, 0, 1, 0,12, 0, 1, 0, 0, 0, 1, 0, \\
&2, 0, 1, 0, 0, 0, 1, 0, 3, 0, 1, 0, 0, 0, 1, 0, 0, 0, 1, 0, 0, 0, 1, 0, 4, 0, 1, 0, 0, 0, 1, 0, \\                                                               &0, 0, 1, 0, 0, 0, 1, 0, 2, 0, 1, 0, 0, 0, 1, 0, 0, 0, 1, 0, 0, 0, 1, 0, 3, 0, 1, 0, 0, 0, 1, 0, \\
&3, 0, 1, 0, 0, 0, 1, 0, 2, 0, 1, 0, 0, 0, 1, 0, 0, 0, 1, 0, 0, 0, 1, 0, 2, 0, 1, 0, 0, 0, 1, 0, \\
&2, 0, 1, 0, 0, 0, 1, 0, 2, 0, 1, 0, 0, 0, 1, 0, 0, 0, 1, 0, 0, 0, 1, 0, 4, 0, 1, 0, 0, 0, 1, 0, \\
&3, 0, 1, 0, 0, 0, 1, 0, 3, 0, 1, 0, 0, 0, 1, 0, 0, 0, 1, 0, 0, 0, 1, 0, 3, 0, 1, 0, 0, 0, 1, 0, \\
&0, 0, 1, 0, 0, 0, 1, 0, 2, 0, 1, 0, 0, 0, 1, 0, 0, 0, 1, 0, 0, 0, 1, 0, 2, 0, 1, 0, 0, 0, 1, 0, \\
&5, 0, 1, 0, 0, 0, 1, 0, 2, 0, 1, 0, 0, 0, 1, 0, 0, 0, 1, 0, 0, 0, 1, 0, 2, 0, 1, 0, 0, 0, 1, 0, \\
&0, 0, 1, 0, 0, 0, 1, 0, 3, 0, 1, 0, 0, 0, 1, 0, 0, 0, 1, 0, 0, 0, 1, 0, 3, 0, 1, 0, 0, 0, 1, 0, \\
&2, 0, 1, 0, 0, 0, 1, 0, 3, 0, 1, 0, 0, 0, 1, 0, 0, 0, 1, 0, 0, 0, 1, 0, 3, 0, 1, 0, 0, 0, 1, 0, \\
&0, 0, 1, 0, 0, 0, 1, 0, 2, 0, 1, 0, 0, 0, 1, 0, 0, 0, 1, 0, 0, 0, 1, 0, 2, 0, 1, 0, 0, 0, 1, 0, \\
&3, 0, 1, 0, 0, 0, 1, 0, 2, 0, 1, 0, 0, 0, 1, 0, 0, 0, 1, 0, 0, 0, 1, 0, 3, 0, 1, 0, 0, 0, 1, 0, \\
&3, 0, 1, 0, 0, 0, 1, 0, 3, 0, 1, 0, 0, 0, 1, 0, 0, 0, 1, 0, 0, 0, 1, 0, 2, 0, 1, 0, 0, 0, 1, 0, \\
&2, 0, 1, 0, 0, 0, 1, 0, 2, 0, 1, 0, 0, 0, 1, 0, 0, 0, 1, 0, 0, 0, 1, 0, 2, 0, 1, 0, 0, 0, 1, 0, \\
&0, 0, 1, 0, 0, 0, 1, 0, 3, 0, 1, 0, 0, 0, 1, 0, 0, 0, 1, 0, 0, 0, 1, 0, 2, 0, 1, 0, 0, 0, 1, 0, \\
&2, 0, 1, 0, 0, 0, 1, 0, 3, 0
\end{align*}
\caption{The first 490 values of the sequence $(\G(0,0,n))_{n \geq 1}$ (the bottom row of Table \ref{T1}).} \label{T.P}
\end{table}

As mentioned, the table suggests that the nim-values of the sequence $(\G(0,0,n))_{n \geq 1}$ is bounded. Our data also shows that the first 1000 values of this sequence is bounded by 12. The pattern of the nim-values may raise question: is the sequence ultimately periodic? We will show that the answer is negative. Moreover, we will show that the periodicity does not even hold for $\P / \N$ status of this sequence.

One may also ask if the positions whose nim-values are 1 in Table \ref{T.P} are periodic. We will prove that periodicity in Section \ref{S.v1}.

%In this section, we prove that the bottom row (similarly, the main diagonal) in Table \ref{T1} is not ultimately periodic.
%Based on that, we conjecture that rows and diagonals in Table \ref{T1} are not ultimately periodic.
Set
\begin{align*}
f(n) =
\begin{cases}
0, & \text{if $(0,0,n)$ is a $\P$-position}; \\
1, & \text{otherwise}.
\end{cases}
\end{align*}

%The following result involves the periodicity of $\P$-positions.

\begin{proposition} \label{P-per}
Let $n \geq 3$.
\begin{enumerate} \itemsep0em
\item [$(1)$] If $n$ is odd, $f(n) = 0$.
\item [$(2)$] If $n \bmod 2^5 \in \{4,12,16,20,28\}$, $f(n) = 0$.
\item [$(3)$] If $n \bmod 2^5 \in \{2,6,8,10, 14, 18, 22, 24, 26, 30\}$, $f(n) = 1$.
\end{enumerate}
%In particular, except for the case $(4)$, the first three cases are periodic with period $p = 32$.
\end{proposition}
\begin{proof}\
\begin{enumerate} \itemsep0em
\item If $n$ is odd, $n = 2l+1$ for some $l$ and so $(0,0,n)$ is a $\P$-position by Theorem \ref{P}.
\item If  $n \bmod 2^5 = 16$ or $n = 2^5m+ 16$, we rewrite $n = 2^4(2m+1)$ and so $(0,0,n)$ is a $\P$-position by Theorem \ref{P}. Set $S = \{4,12,20,28\}$. Note that each element $s \in S$ can be written as $2^2(2l+1)$. If $n \bmod 2^5 \in S$, we rewrite $n = 2^5m + 2^2(2l+1) = 2^2(2^3m+2l+1)$ and so $(0,0,n)$ is a $\P$-position by Theorem \ref{P}.
\item Set $T = \{2,6,8,10, 14, 18, 22, 24, 26, 30\}$. Note that every $t \in T$ can be written as $2^{2k+1}(2l+1)$ with $k \leq 1$. Therefore, if $n \bmod 2^5 \in T$, $n = 2^{2k+1}(2^{4-2k}m+2l+1)$. Since $4-2k > 0$, $n = 2^{2k+1}(2l'+1)$ and so $(0,0,n)$ is not a $\P$-position by Theorem \ref{P}.
\end{enumerate}
\end{proof}

We display the first 490 values of $f(n)$ as in Table \ref{T.f}, in which each row has $2^5$ values (except for the first row) and the values in the  bold column (the 10th one from the right) are corresponding to those $f(n)$ with $n \bmod 2^5 = 1$. Note that Table \ref{T.f} is obtained from Table \ref{T.P} by replacing positive values by 1.

\begin{table}[ht]
\begin{align*}
&{\bf0},0,1,0,0,0,1,0,1,0, 1,0,0,0,1,0,0,0,1,0,0,0,1,0,1,0,1,0,0,0,1,0, \\
&{\bf1},0,1,0,0,0,1,0,1,0, 1,0,0,0,1,0,0,0,1,0,0,0,1,0,1,0,1,0,0,0,1,0, \\
&{\bf0},0,1,0,0,0,1,0,1,0, 1,0,0,0,1,0,0,0,1,0,0,0,1,0,1,0,1,0,0,0,1,0, \\
&{\bf1},0,1,0,0,0,1,0,1,0, 1,0,0,0,1,0,0,0,1,0,0,0,1,0,1,0,1,0,0,0,1,0, \\
&{\bf1},0,1,0,0,0,1,0,1,0, 1,0,0,0,1,0,0,0,1,0,0,0,1,0,1,0,1,0,0,0,1,0, \\
&{\bf1},0,1,0,0,0,1,0,1,0, 1,0,0,0,1,0,0,0,1,0,0,0,1,0,1,0,1,0,0,0,1,0, \\
&{\bf0},0,1,0,0,0,1,0,1,0, 1,0,0,0,1,0,0,0,1,0,0,0,1,0,1,0,1,0,0,0,1,0, \\
&{\bf1},0,1,0,0,0,1,0,1,0, 1,0,0,0,1,0,0,0,1,0,0,0,1,0,1,0,1,0,0,0,1,0, \\
&{\bf0},0,1,0,0,0,1,0,1,0, 1,0,0,0,1,0,0,0,1,0,0,0,1,0,1,0,1,0,0,0,1,0, \\
&{\bf1},0,1,0,0,0,1,0,1,0, 1,0,0,0,1,0,0,0,1,0,0,0,1,0,1,0,1,0,0,0,1,0, \\
&{\bf0},0,1,0,0,0,1,0,1,0, 1,0,0,0,1,0,0,0,1,0,0,0,1,0,1,0,1,0,0,0,1,0, \\
&{\bf1},0,1,0,0,0,1,0,1,0, 1,0,0,0,1,0,0,0,1,0,0,0,1,0,1,0,1,0,0,0,1,0, \\
&{\bf1},0,1,0,0,0,1,0,1,0, 1,0,0,0,1,0,0,0,1,0,0,0,1,0,1,0,1,0,0,0,1,0, \\
&{\bf1},0,1,0,0,0,1,0,1,0, 1,0,0,0,1,0,0,0,1,0,0,0,1,0,1,0,1,0,0,0,1,0, \\
&{\bf0},0,1,0,0,0,1,0,1,0, 1,0,0,0,1,0,0,0,1,0,0,0,1,0,1,0,1,0,0,0,1,0, \\
&{\bf1},0,1,0,0,0,1,0,1,0.
% 1,0,0,0,1,0,0,0,1,0
\end{align*}
\caption{The first 490 values of the sequence $(f(n))_{n \geq 1}$} \label{T.f}
\end{table}

One may ask if the sequence $(f(n))_{n \geq 0}$ is ultimately periodic. The exceptional case not included in Proposition \ref{P-per} is $n \bmod 2^5 \equiv 0$ whose $f(n)$ values form the bold column (first column) in Table \ref{T.f}. This is the case the periodicity dies out. We now discuss this phenomenon.

\begin{theorem} \label{P-n-per}
The sequence $(f(n))_{n \geq 0}$ is not ultimately periodic.
\end{theorem}

\begin{proof}
Assume on the contradiction that $(f(n))_{n \geq 0}$ is ultimately periodic with period $p$ and preperiod $n_0$, meaning $f(n+p) = f(n)$ for all $n \geq n_0$. It follows from Proposition \ref{P-per} that $p$ is even. Otherwise, we can choose even $n$ large enough such that $(n+p) \bmod{2^5}$ belongs to the modulo set in the case (3) of Proposition \ref{P-per}, giving a contradiction.

Set $S_1 = \{4,12,16,20,28\}$, $S_2 = \{2,6,8,10, 14, 18, 22, 24, 26, 30\}$, and $S = \{2,4,6,8,10,12,14,18,20,22,26\}$.
Note that $S =  \{|s_1-s_2| \mid s_1 \in S_1, s_2 \in S_2, s_1 > s_2\}$. We next prove that $p \bmod {2^5} \notin S$. Assume on the contradiction that $p \bmod 2^5 \in S$. Without loss of generality, we can assume that $p \bmod 2^5 = s_2 - s_1$ for some $s_1 \in S_1$ and $s_2 \in S_2$. Since the sequence $(f(n))_{n \geq 0}$ is ultimately periodic, we can choose $n$ large enough such that $f(n+p) = f(n)$. Moreover, we can choose $n$ such that $n \bmod 2^5 = s_1$, giving $f(n) = 0$ by  Proposition \ref{P-per}. We now have $(n+p) \bmod 2^5 = s_2$ and so $f(n+p) = 1$ by  Proposition \ref{P-per}, giving $f(n+p) \neq f(n)$. This is a contradiction.

We next show that $p \bmod 2^5 \notin \{0, 16, 24, 28, 30\} = T$. We show that for any $t \in T$ such that $p \bmod 2^5 = t$, there exists $n \geq n_0$ such that $n \bmod 2^5 = 1$ and $f(n+p) \neq f(n)$.

Consider the case $p \bmod 2^5 = 0$. We write $p = 2^r(2q+1)2^5$. If $r$ is odd, we choose $n = 2^s(2t+1)2^5$ for some even $s > r$ such that $n \geq n_0$. This task is simply increasing $s$ until $n$ exceeds $n_0$. Then $(0,0,n)$ is not a $\P$-position by Theorem \ref{P} and so $f(n) = 1$. Note that $n+p = 2^{5+r}(2^{s-r}(2t+1)+2q+1)$. Since $5+r$ is even and $s > r$, $(0,0,n)$ is a $\P$-position by Theorem \ref{P} and so $f(n+p) = 0$. Similarly, if $r$ is even, we choose $n$ with similar form and odd $s$. The same argument gives $f(n) = 0$ and $f(n+p) = 1$. In both cases, we have $f(n+p) \neq f(n)$, contradicting to the ultimate periodicity of $(f(n))_{n \geq 0}$.

Similarly, one can argue that $p \bmod 2^5 \neq 16, 24, 28, 30$. Thus, $p \bmod 2^5 \notin \{0,1,2, \ldots, 31\}$. This is a contradiction. Therefore, the sequence $(f(n))_{n \geq 0}$ is not ultimately periodic.
\end{proof}

\begin{remark}
If the bottom row of Table \ref{T1} is ultimately periodic, $(f(n))_{n \geq 1}$ is ultimately periodic.
It follows from Theorem \ref{P-n-per} that the bottom row (and similarly, the main diagonal) of Table \ref{T1} is not ultimately periodic.
%although its sequence members corresponding to even $b$ are all zero.
Based on this phenomenon, it is also reasonable to conjecture that
\end{remark}

\begin{conjecture}
Rows of Table \ref{T1} are not ultimately periodic.
\end{conjecture}

The reason for the conjecture is that computing sequence $(\G(0,a,n))_{n \geq a}$ partially requires data from sequences $(\G(0,c,n))_{n \geq c}$ for $c < a$. This is because from any position $(0,a,n)$ for $n \geq 2a$, there exists a move to some position $(0,c,m)$, by moving $a-c$ tokens from the pile of size $n$ to the pile of size $1$.

We now discuss the periodicity of $\P$-positions. We say that $g$-positions in the $a^{th}$ row of Table \ref{T1} is ultimately periodic if there exist $n_0$ and $p$ such that for all $n \geq n_0$, $(0,a,n)$ is a $g$-position if and only if $(0,a,n+p)$ is a $g$-position. It follows from Theorem \ref{P-n-per} that

\begin{corollary} \label{P-notperiodic}
The distribution of $\P$-positions on the bottom row are not periodic.
\end{corollary}

%%%%%%%%%%%%%%%%%%%%%%%%%%%%%%%%%%%%%%%%%%%%%%%%%%
%%%%%%%%%%%%%%%%%%%%%%%%%%%%%%%%%%%%%%%%%%%%%%%%%%
% 1-POSITIONS

\section{A formula for 1-positions} \label{S.v1}

We prove a formula for $1$-positions.

\begin{theorem} \label{v1}
The $1$-positions form the set
\[\B = \{(0,0,4k+2),(0,4k+2,4k+2), (0,2,4k+1), (0,4l-1,4l+1) | k \geq 0, l \geq 1\}.\]
\end{theorem}
\begin{proof}
Recall that the $\P$-positions form the set $\P = \{(0,0,0), (0,0,2^{2k}(2l+1)), (0,2^{2k}(2l+1),2^{2k}(2l+1)) \mid  k \geq 0, l\geq 0\}$ (Theorem \ref{P}). It can be checked that $\P \cap \B = \emptyset$.
We need to prove the following two facts:
\begin{itemize} \itemsep0em
\item [\rm{(i)}] there is no move between any two distinct positions in $\B$, and
\item [\rm{(ii)}] from any position not in $\P \cup \B$, there exists a move that terminates in $\B$.
\end{itemize}
It is simple to check \rm{(i)} and we leave to the reader. For \rm{(ii)}, let $x = (0,a,b) \notin \P \cup \B$. We write $a = 4q+r$ and $b = a + 4s+t$. By Corollary \ref{C.iso}, it suffices to consider the case where $a \leq b/2$ or equivalently $4s+t \geq 4q+r$.

We have mentioned that in position $(a,b,c)$, we assume that $a \leq b \leq c$. Note that after a move, the hierarchy may not be the same and we need reorder the entries. We skip the process of reordering entries in a position after a move in this proof. For example, we may write ``removing $3$ token from the third entry to the first entry leads $(1,2,7)$ to $(0,2,2)$". Here, we mean $(0,2,2) \equiv (4,2,4)$ after subtracting $2$ from each entry and reordering entries by increasing order.

\item [(\rm{I})] Consider the case $t=0$.
\begin{enumerate} \itemsep0em
\item If $r = 0$, $x = (0,4q,4q+4s)$. We examine the parity of $s$.
    \begin{enumerate}
    \item If $s$ is odd or equivalently $s = 2p+1$, moving $2s$ tokens from the third pile to the second leads $x$ to $(0,4q+2s,4q+2s) = (0,4(q+p)+2,4(q+p)+2) \in B$.
    \item If $s$ is even or equivalently $s = 2p$, moving $2s-1$ tokens from the third pile to the second leads $x$ to $(0,4q+2s-1,4q+2s+1) = (0,4(q+p)-1,4(q+p)+1) \in \B$.
    \end{enumerate}
\item If $r = 1$, $x = (0,4q+1,4q+1+4s)$  with $s > q$. Moving $4q+3$ tokens from the third pile to the first leads $x$ to $(4q+3,4q+1,4s-2) \equiv (0,2,4(s-q-1)+1) \in \B$.
\item If $r = 2$, $x = (0,4q+2, 4q+2+4s)$ with $s > q$. We examine the parity for $s$.
    \begin{enumerate}
    \item If $s$ is odd or equivalently $s = 2p+1$, moving $4p+1$ tokens from the third pile to the second leads $x$ to $(0,4(q+p+1)-1,4(q+p+1)+1) \in B$.
    \item If $s$ is even or equivalently $s = 2p$, moving $2s = 4p$ tokens from the third pile to the second leads $x$ to $(0,4(q+p)+2,4(q+p)+2) \in \B$
    \end{enumerate}
\item If $r = 3$, $x = (0,4q+3,4q+3+4s)$ with $s > q$. Moving $4q+1$ tokens from the third entry to the first leads $x$ to $(4q+1, 4q+3, 4s+2) \equiv (0,2,4(s-q)+1) \in \B$.
\end{enumerate}

\item [(\rm{II})] Consider the case $t=1$.
\begin{enumerate}  \itemsep0em
\item If $r = 0$, $x = (0,4q,4q+4s+1)$ with $s \geq q$. Moving $4q-2$ tokens from the third entry to the first leads $x$ to $(4q-2,4q,4s+3)) \equiv (0,2,4(s+1-q)+1) \in \B$.
\item If $r = 1$, $x = (0,4q+1, 4q+1+4s+1)$ with $s \geq q$. We examine the parity of $s-q$.
    \begin{enumerate}
    \item If $s-q$ is odd, $s > q$. Moreover, $s-q-1 = 2p$ for some $p \geq 0$. Moving $2q+2s+1$ tokens from the third entry to the first leads $x$ to $(2q+2s+1,4q+1,2q+2s+1) = (2(s-q-1)+2,0,2(s-q-1)+2) = (0,4p+2,4p+2) \in \B$.
    \item If $s-q$ is even, $s-q = 2p$ for some $p \geq 0$. Moving $2q+2s$ tokens from the third entry to the first leads $x$ to $(2q+2s,4q+1,2q+2s+2) \equiv (2(s-q)-1,0,2(s-q)+1) = (0,4p-1,4p+1) \in \B$.
    \end{enumerate}
\item If $r = 2$, $x = (0,4q+2,4q+2+4s+1)$ with $s > q$. Moving $4q+4$ tokens from the third entry to the first leads $x$ to $(4q+4,4q+2,4s-1) \equiv (0,2,4(s-1-q)+1) \in \B$.
\item If $r = 3$, $x = (0,4q+3,4q+3+4s+1)$ with $s > q$. Moving $4q+3$ tokens from the third entry to the first leads $x$ to $(4q+3,4q+3,4s+1) \equiv (0,0,4(s-q-1)+2) \in \B$.
\end{enumerate}

\item [(\rm{III})] Consider the case $t = 2$. %Moving $a-1$ tokens from the third entry to the first one results in the position $(1,1,4s+3)$ (after subtracting $a-1$ from all three entries).
\begin{enumerate}  \itemsep0em
\item If $r = 0$, $x = (0,4q,4q+4s+2)$ with $s \geq q$. Moving $4q$ tokens from the third entry to the first leads $x$ to $(4q,4q,4s+2) \equiv (0,0,4(s-q)+2) \in \B$.
\item If $r = 1$, $x = (0,4q+1,4q+1+4s+2)$ with $s \geq q$. Moving $4q-1$ tokens from the third entry to the first leads $x$ to $(4q-1,4q+1,4(s+1)) \equiv (0,2,4(s-q+1)+1) \in \B$.
\item If $r = 2$, $x = (0,4q+2,4q+2+4s+2)$. If $s = q$, moving $4q+1$ tokens from the third entry to the first leads $x$ to $(4q+1,4q+2,4s+3) \equiv (0,1,2) \in B$. We now consider the case $s > q$ and examine $s-q$.
    \begin{enumerate}
    \item If $s-q = 2p$ for some $p$. Moving $2q+2s+1$ tokens from the third entry to the first leads $x$ to $(2q+2s+1, 4q+2,2q+2s+3) \equiv (0, 2(s-q)-1, 2(s-q)+1) = (0,4p-1,4p+1) \in \B$.
    \item If $s-q = 2p+1$, moving $2q+2s+2$ tokens from the third entry to the first leads $x$ to $(2q+2s+2,4q+2,2q+2s+2) \equiv (0,4p+2,4p+2) \in \B$.
    \end{enumerate}
\item If $r = 3$, $x = (0,4q+3,4q+3+4s+2)$ with $s > q$. If $s$ is even or equivalently $s = 2p$, moving $2s = 4p$ tokens from the third entry to the second leads $x$ to $(0,4(q+p)+3,4(q+p)+5) \equiv (0,4(q+p+1)-1, 4(q+p+1)+1) \in \B$. If $s$ is odd or $s = 2p+1$, moving $2s +1 = 4p+3$ tokens from the third entry to the second leads $x$ to $(0,4(q+p+1)+2,4(q+p+1)+2) \in \B$.
\end{enumerate}

\item [(\rm{IV})] Consider the case $t=3$.
\begin{enumerate}  \itemsep0em
\item If $r = 0$, $x = (0,4q,4q+4s+3)$ with $s \geq q$. Moving $4q+2$ tokens from the third entry to the first leads $x$ to $(4q+2,4q,4s+1) \equiv (0,2,4(s-q)+1) \in \B$.
\item If $r = 1$, $x = (0,4q+1,4q+1+4s+3)$ with $s \geq q$. Moving $4q+1$ tokens from the third entry to the first leads $x$ to $(4q+1,4q+1,4s+3) \equiv (0,0,4(s-q)+2) \in \B$.
\item If $r = 2$, $x = (0,4q+2,4q+2+4s+3)$ with $s \geq q$. Moving $4q$ tokens from the third entry to the first leads $x$ to $(4q,4q+2,4s+5) \equiv (0,2,4(s-q+1)+1) \in \B$.
\item If $r = 3$, $x = (0,4q+3,4q+3+4s+3)$ with $s \geq q$.
    \begin{enumerate}   \itemsep0em
    \item Consider the case $s = q$.
    %Then $x =  (1,4q,4q+4(q-1)+3) = (1,4q,8q-1)$.
    Moving $4q+2$ tokens from the third entry to the first leads $x$ to $(4q+2,4q+3,4q+4) \equiv (0,1,2) \in \B$.
    \item Consider the case $s > q$.
    \begin{enumerate}
    \item If $s-q$ is odd, or $s-q = 2p+1$ for some $p$. Moving $2q+2s+3$ tokens from the third entry to the first leads $x$ to $(2q+2s+3, 4q+3, 2q+2s+3) \equiv (0,4p+2,4p+2) \in \B$.
    \item If $s-q$ is even, or $s-q = 2p$ for some $p$. Moving $2q+2s+2$ tokens from the third entry to the first leads $x$ to $(2q+2s+2, 4q+3, 2q+2s+4) \equiv (2(s-q)-1,0,2(s-q)+1) = (0, 4p-1,4p+1) \in \B$.
    \end{enumerate}
    \end{enumerate}
\end{enumerate}
\end{proof}

As mentioned in Section \ref{S.P.per}, the positions whose nim-values are 1 appearing in Table \ref{T.P} are periodic, corresponding to those having formula $(0,0,4k+2)$. Moreover, by Theorem \ref{v1}, we have the following periodicity of 1-positions on the bottom and the third one from the bottom

\begin{corollary}
The 1-positions distributed on the bottom row $($reps.~main diagonal$)$ and the third row $($reps.~third diagonal from the top diagonal$)$ from the bottom of Table \ref{T1} are periodic.
\end{corollary}

%%%%%%%%%%%%%%%%%%%%%%%%%%%%%%%%%%%%%%%%%%%%%%%%%%
%%%%%%%%%%%%%%%%%%%%%%%%%%%%%%%%%%%%%%%%%%%%%%%%%%
\section{An estimation of the distribution of nim-values} \label{S.Dis}

In this section, we will discuss the bound of the distribution of $g$-positions for each $g$.

We have shown that the $\P$-positions appear only on the bottom row (and the main diagonal) of Table \ref{T1}.  By Theorem \ref{v1}, 1-positions are distributed on only bottom row and the third row (from the bottom), except for the position $(0,1,2)$. Let we ignore diagonals as they are identical to rows. Although 1-positions are periodic on each row, the $\P$-positions are not periodic (Corollary \ref{P-notperiodic}). We haven't known if the periodicity will hold for $k$-positions on each row (for $k \geq 2$).

Our computation shows that for each nim-value $g \geq 2$, $g$-positions are distributed on the bottom $2g$ rows or less of Table \ref{T1} (or top $2g$ diagonals).

 \begin{conjecture}
 For $a \leq b/2$, if $\G(0,a,b) = g \geq 2$ then $a \leq 2g-1$.
 \end{conjecture}

We give here examples with $g = 2$ and $g = 3$. For $g = 2$, 2-positions $(0,a,b)$  has $a \leq 3$ (We have tested for $a \leq 14$) with following first 300 values for $b$:
\begin{enumerate} \itemsep0em
\item [] $a = 0$: 32, 72, 104, 120, 128, 136, 200, 216, 232, 248, 288.
\item [] $a = 1$: 3, 13, 17, 18, 23, 29, 33, 38, 39, 40, 46, 48, 53, 55, 57, 61, 62, 68, 70, 80, 85, 86, 87, 91, 95, 96, 100, 101, 102, 109, 110, 116, 117, 118, 124, 126, 128, 132, 136, 144, 145, 150, 151, 159, 165, 166, 167, 174, 176, 181, 182, 187, 189, 195, 196, 197, 204, 206, 208, 212, 213, 214, 222, 223, 224, 228, 229, 238, 240, 244, 256, 263, 278, 279, 284, 285, 294, 295, 296.
\item [] $a = 2$: 4, 6, 7, 11, 26, 27, 46, 71, 78, 80, 110, 136, 151, 159, 160, 174, 182, 190, 232, 247, 248, 256, 263, 264, 271, 272, 279, 288.
\item [] $a = 3$: 23, 67, 147, 243, 259, 275
\end{enumerate}

For $g = 3$, 3-positions $(0,a,b)$ has $a \leq 3$ (We have tested for $a \leq 14$) with following first 300 values for $b$:
\begin{enumerate} \itemsep0em
\item [] $a = 0$: 8, 40, 88, 96, 160, 168, 184, 264, 280, 296.
\item [] $a = 1$: 4, 5, 21, 25, 27, 31, 37, 45, 47, 54, 69, 76, 77, 78, 83, 84, 89, 111, 115, 125, 130, 131, 141, 146, 147, 156, 160, 161, 173, 192, 193, 198, 199, 207, 216, 221, 227, 237, 242, 253, 257, 259, 267, 269, 274, 289, 299.
\item [] $a = 2$: 18, 19, 22, 31, 35, 48, 55, 62, 63, 70, 95, 102, 103, 104, 118, 119, 123, 134, 138, 139, 147, 150, 154, 167, 168, 170, 183, 184, 187, 199, 207, 214, 216, 224, 231, 235, 246, 250, 251, 260, 267, 275, 282, 283, 287, 292.
\item [] $a = 3$: 7, 10, 36, 53, 60, 68, 89, 92, 107, 108, 124, 179, 184, 204, 221, 229, 236, 247, 279, 292.
\end{enumerate}

Note that the estimation of $a$ with $a \leq 2g-1$ seems to be overestimation. It seems that when $g$ increases, $a$ gets closed to $g$ than to $2g-1$. For example, up to $b = 200$, we have following bound for $a$:
\begin{enumerate} \itemsep0em
\item [] $g = 4$, $a \leq 6$;
\item [] $g = 5$, $a \leq 6$;
\item [] $g = 6$, $a \leq 7$;
\item [] $g = 7$, $a \leq 9$;
\item [] $g = 8$, $a \leq 12$;
\item [] $g = 9$, $a \leq 11$;
\item [] $g = 10$, $a \leq 13$.
\end{enumerate}

%%%%%%%%%%%%%%%%%%%%%%%%%%%%%%%%%%%%%%%%%%%%%%%%%%
%%%%%%%%%%%%%%%%%%%%%%%%%%%%%%%%%%%%%%%%%%%%%%%%%%

\section{Discussion}
We have analyzed 3-pile {\sc Sharing Nim} and shown that this game can be solved in quadratic time. Especially, the strategy of this game can be understood easily by first year students at high school. We have also obtained a formula for 1-positions. We have discovered the case in which finite nim-sequences form palindromes when reading horizontally or vertically. Recall that palindromes is not unusual for two-pile Nim-like games when we read nim-values on the line that orthogonal to the main diagonal, consisting of entries $(e_{a-i,i})$ for given $a$ and $0 \leq i \leq a$. Such an example is Wythoff'g game \cite{Wyt}.

We have conjectured that the nim-sequence $(\G(0,a,n))_{n \geq a}$ is bounded, presenting the possibility of a surprising pattern of nim-sequences of variants of {\sc Nim}: nim-values are bounded regardless of the increase of a single pile. Regardless of the tendency of being bounded, it seems that these nim-sequences are not ultimately periodic, with the case $a = 0$ being verified.

A possible research direction is to study the game with more than 3 piles. It can be seen that results in Section \ref{S.palindrome} still holds for {\sc Sharing Nim} with more than 3 piles, using Figure \ref{F.com}.

Another research direction is to playing {\sc Sharing Nim} on graphs. The idea is as follow. On a direct graph, we place on each vertex a number of tokens. Each move consists of choosing one vertex $x$ and moving an arbitrary number of tokens from that vertex to an adjacent vertex $y$  provided that vertex $y$ does not have more tokens than $x$ after the move. Let we call this game {\sc Graph Nim}. From this view, playing $n$-pile {\sc Sharing Nim} is similar to playing {\sc Graph Nim} on a complete graph $K_n$.

%%%%%%%%%%%%%%%%%%%%%%%%%%%%%%%%%%%%%%%%%%%%%%%%%%
%%%%%%%%%%%%%%%%%%%%%%%%%%%%%%%%%%%%%%%%%%%%%%%%%%

\bibliographystyle{amsplain}

\bigskip
\bigskip

{\bf Appendix. Data on the bounded nim-sequences} \label{App}

\begin{figure}[ht]
\begin{center}
\includegraphics[height=5.5cm]{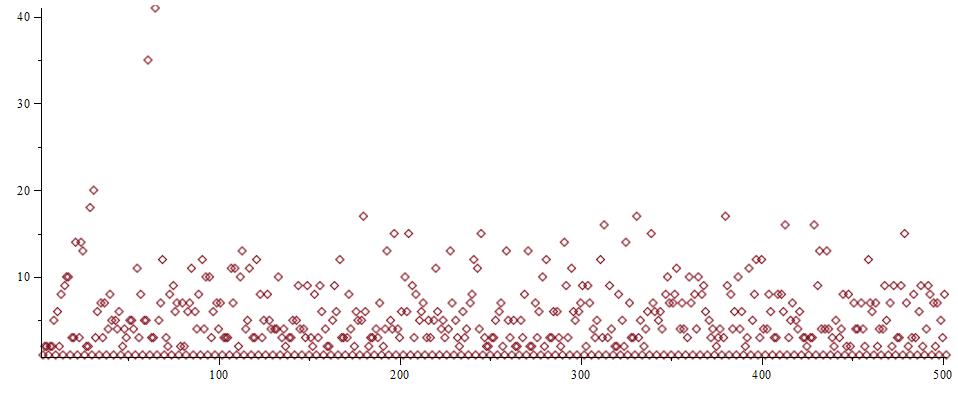}
\caption{The first 500 nim-values of the sequence $(\G(1,3,n))_{n \geq 3}$} \label{row3}
\end{center}
\end{figure}

\begin{figure}[ht]
\begin{center}
\includegraphics[height=5.5cm]{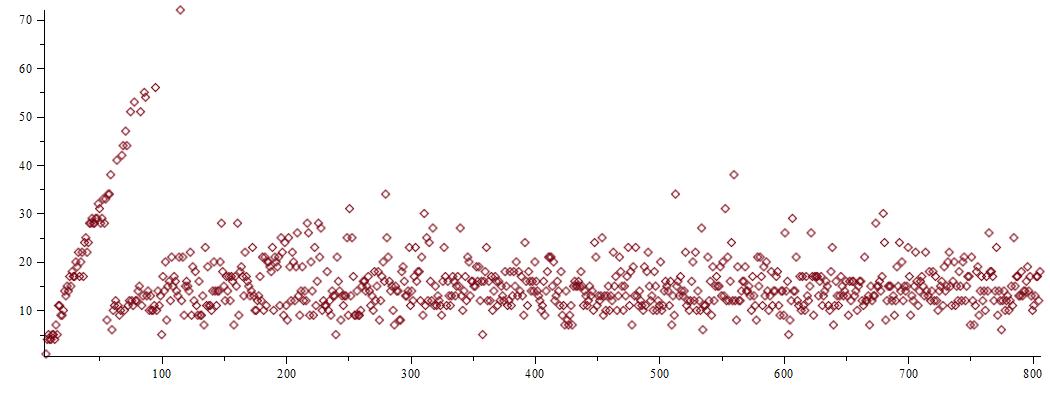}
\caption{The first 800 nim-values of the sequence $(\G(1,7,n))_{n \geq 7}$} \label{row7}
\end{center}
\end{figure}

\end{document}